\documentclass[11pt]{amsart}
\usepackage{amssymb,color}
\usepackage[mathscr]{euscript}
\usepackage{amsmath}
\usepackage{latexsym}
\usepackage{verbatim}
\usepackage[left=1in,right=1in,top=1in,bottom=1in]{geometry}
\usepackage[linktocpage=true]{hyperref}

\hypersetup{ colorlinks   = true, 
 urlcolor  = blue, 
 linkcolor    = black, 
 citecolor   = red 
 }

\definecolor{purple}{rgb}{.5,0,1}

\newcommand{\bbC}{{\mathbb C}}

\newcommand{\bbP}{{\mathbb P}}

\newcommand{\bbR}{{\mathbb R}}

\newcommand{\bbZ}{{\mathbb Z}}

\newcommand{\tr}{{\mathrm{tr}}}

\newcommand{\sW}{{\mathscr{W}}}

\newcommand{\dd}{{\mathrm{d}}}

\newcommand{\SL}{{\mathrm{SL}}}

\renewcommand{\Im}{{\mathrm{Im}}}

\DeclareMathOperator{\supp}{{supp}}

\newtheorem{theorem}{Theorem}[section]

\theoremstyle{definition}

\theoremstyle{definition}

\newtheorem{lemma}[theorem]{Lemma}

\numberwithin{theorem}{section}
\numberwithin{equation}{section}

\begin{document}

\author[V.\ Bucaj]{Valmir Bucaj}
\address{Department of Mathematics, United States Military Academy, West Point, NY~10996, USA}
\email{valmir.bucaj@westpoint.edu}
\thanks{V.B.\ and V.G.\ were supported in part by NSF grant DMS--1361625.}

\author[D.\ Damanik]{David Damanik}
\address{Department of Mathematics, Rice University, Houston, TX~77005, USA}
\email{damanik@rice.edu}
\thanks{D.D.\ was supported in part by NSF grants DMS--1361625 and DMS--1700131.}

\author[J.\ Fillman]{Jake Fillman}
\address{Department of Mathematics, Virginia Tech, 225 Stanger Street, Blacksburg, VA~24061, USA}
\email{fillman@vt.edu}
\thanks{J.F.\ was supported in part by an AMS-Simons travel grant, 2016--2018}

\author[V.\ Gerbuz]{Vitaly Gerbuz}
\address{Department of Mathematics, Rice University, Houston, TX~77005, USA}
\email{vitaly.gerbuz@rice.edu}

\author[T.\ VandenBoom]{Tom VandenBoom}
\address{Department of Mathematics, Yale University, New Haven, CT~06511, USA}
\email{thomas.vandenboom@yale.edu}
\thanks{T.V.\ was supported in part by an AMS-Simons travel grant, 2018-2020.}

\author[F.\ Wang]{Fengpeng Wang}
\address{School of Mathematics, Ocean University of China, Qingdao, China 266100}
\email{wfpouc@163.com}
\thanks{F.W.\ was supported by CSC (No.201606330003) and NSFC (No.11571327).}

\author[Z.\ Zhang]{Zhenghe Zhang}
\address{Department of Mathematics, UC Riverside, Riverside, CA~92521, USA}
\email{zhenghe.zhang@ucr.edu}
\thanks{Z.Z.\ was supported in part by NSF grant DMS--1764154.}

\title[Positive Lyapunov Exponents and LDT for Continuum Anderson Models]{Positive Lyapunov Exponents and a Large Deviation Theorem for Continuum Anderson Models, Briefly}

\begin{abstract}
In this short note, we prove positivity of the Lyapunov exponent for 1D continuum Anderson models by leveraging some classical tools from inverse spectral theory. The argument is much simpler than the existing proof due to Damanik--Sims--Stolz, and it covers a wider variety of random models.  Along the way we note that a Large Deviation Theorem holds uniformly on compacts.
\end{abstract}

\maketitle

\vspace{-1cm}
\section{Introduction}
\subsection{Background}
It is well understood that random Schr\"odinger operators in one space dimension exhibit Anderson localization!

While this introductory statement is correct in many ways, it is nevertheless important to clarify what is actually meant. Does one talk about spectral localization or dynamical localization? Does one consider the discrete setting or the continuum setting? Even if one considers the (easier) discrete setting, is the assertion made for the standard model, or for more general models such as the ones considered in \cite{DSS2004}? What is assumed about the single-site distribution?

It is true that no matter how one answers these questions, localization is indeed known. However, the difficulty of the known proofs depends heavily on the answers. For example, the proofs are short and elegant in the discrete setting with an absolutely continuous single-site distribution, but they can be quite difficult once the continuum setting and/or singular single-site distributions are considered.

Some of the landmark papers are Kunz-Souillard \cite{KS80} (standard discrete model with an absolutely continuous single-site distribution), Carmona-Klein-Martinelli \cite{CKM87} (standard discrete model with a general single-site distribution), and Damanik-Sims-Stolz \cite{DSS2002Duke} (standard continuum model with a general single-site distribution).

In the case of a general single-site distribution, the localization proof typically consists of two steps. First, one proves that the Lyapunov exponent is positive for a sufficiently large set of energies by an application of F\"urstenberg's theorem, and second, one parlays this positivity statement into the exponential decay of generalized eigenfunctions, showing in effect that the spectrum is pure point and the eigenfunctions decay exponentially. A second look at the structure of the eigenfunctions then allows one to control their semi-uniform localization properties, which in turn yields dynamical localization. Traditionally, this second step was performed via multi-scale analysis.

The first step is very easy to implement in the discrete case and the Lyapunov exponent turns out to be positive for every energy via a straightforward verification of the assumptions of F\"urstenberg's theorem. In the continuum case, on the other hand, verifying these assumptions is less straightforward; it is was accomplished in \cite{DSS2002Duke} away from a discrete set of energies via inverse scattering theory.

In scenarios where the initial proofs were involved, it is of interest to find simplifications of the arguments. For the standard discrete model with a general single-site distribution studied in \cite{CKM87}, there have been several recent papers proposing such simplifications \cite{BDFGVWZ, GK18, JZ18}. These new proofs simplify the second step in the two-step procedure described above (since the first step cannot be simplified, as indicated above).

In this paper we take a new look at the first step in this procedure for continuum models. Rather than using inverse scattering theory we will use inverse spectral theory, and the resulting proof of positive Lyapunov exponents for energies outside a discrete set turns out to be significantly simpler. Our setting is also more general than that of \cite{DSS2002Duke}, so that technically speaking, we generalize the scope of the approach.

We also discuss a Large Deviation Theorem (LDT), which demonstrates that the second step in the localization proof for continuum models can then be carried out in complete analogy to the treatment of the discrete case developed in \cite{BDFGVWZ}. Since this is entirely straightforward, we do not carry this out explicitly, but merely note that with the present work and \cite{BDFGVWZ}, both steps in the two-step procedure to prove localization for 1D continuum Anderson models have been simplified.

\subsection{Main Result}
Fix two parameters $0 < \delta \leq m$, and define
\[
\sW
=
\bigcup_{\delta \leq s \leq m} L^2[0,s).
\]
To distinguish the fibers, let us denote the length of the domain by $s = \ell(f)$ whenever $f \in L^2[0,s)$. We specify a continuum Anderson model by choosing a probability measure $\widetilde\mu$ on $\sW$ such that
\begin{equation}\label{cdn:mutildeUnifL2bd}\tag{$\widetilde{\mu}$\,Bd}
\widetilde\mu\text{-}\mathrm{ess} \sup \|f\|_{L^2} < \infty.
\end{equation}
We naturally obtain the full shift
\[
\Omega = \sW^{\bbZ}, \quad \mu = \widetilde{\mu}^{\bbZ}, \quad
[T\omega]_n = \omega_{n+1}.
\]
Then, for each $\omega \in \Omega$, we obtain a potential $V_\omega$ by concatenating $\ldots,\omega_{-1},\omega_0,\omega_1,\ldots$, and an associated Schr\"odinger operator $H_\omega = -\partial_x^2 + V_\omega$.  More specifically, define
\begin{equation} \label{eq:concatEndptDef}
s_n
=
s_n(\omega)
:=
\begin{cases}
\sum_{j=0}^{n-1} \ell(\omega_j) & n \geq 1 \\
0 & n = 0 \\
-\sum_{j=n}^{-1} \ell(\omega_j) & n \leq -1,
\end{cases}
\end{equation}
denote $I_n = [s_{n},s_{n+1})$, and define
\begin{equation}\label{eq:concatDef}
V_\omega(x)
=
\omega_n(x - s_{n}),\quad
\text{ for each } x \in I_n.
\end{equation}

Let us note that this setting, which is related to those considered in \cite{dfg, DSS2004}, includes that of \cite{DSS2002Duke} as a special case.


For each $w \in \sW$, $E \in \bbC$,  $A^E(w)$ is the unique $\SL(2,\bbC)$ matrix with
\begin{align}\label{eq:cocycle}
\begin{bmatrix}
\psi(s_1) \\ \psi'(s_1)
\end{bmatrix}
=
A^E(w)
\begin{bmatrix}
\psi(0) \\ \psi'(0)
\end{bmatrix}
\end{align}
whenever $H_\omega \psi = E \psi$ with $\omega_0 = w$. For each $E$, the \emph{Lyapunov exponent} is given by
\[
L(E)
=
\lim_{n\to\infty} \frac{1}{n} \int_\Omega  \log\|A^E_n(\omega) \| \, d\mu(\omega),
\quad
\text{where } A^E_n(\omega) = A^E(\omega_{n-1}) \cdots A^E(\omega_0), \ n \geq 1.
\]

One obvious obstruction to localization is if all elements of the support of $\widetilde\mu$ commute in the free product sense, so that one cannot distinguish permutations of elements of the support after concatenation. When this is the case, all realizations $V_\omega$ are periodic, and localization clearly fails. This is the only obstruction to localization; we formulate the negation of this as our nontriviality condition.  For $f_j \in L^2[0,a_j)$, $j=1,2$, we write
\[
(f_1\star f_2)(x)
=
\begin{cases}
f_1(x) & 0 \le x < a_1 \\
f_2(x-a_1) & a_1 \le x < a_1+a_2.
\end{cases}.
\]
The \emph{nontriviality condition} is then the following:
\begin{equation}\tag{NC}\label{eq:NCCAMdef}
\text{ There exist } f_j \in \supp \widetilde\mu \text{ such that }
f_1\star f_2 \neq f_2 \star f_1.
\end{equation}
Let us note that the equality that fails in \eqref{eq:NCCAMdef} is in $L^2$, so we really mean that $f_1\star f_2$ and $f_2 \star f_1$ differ on a set of positive Lebesgue measure.

\begin{theorem} \label{t:main}
If $\widetilde\mu$ satisfies \eqref{eq:NCCAMdef} and \eqref{cdn:mutildeUnifL2bd}, then there is a discrete set $D \subset \bbR$ such that the Lyapunov exponent is positive away from $D$:
\begin{align*}
L(E) > 0 \; \text{ for all } \, E \in \bbR\setminus D.
\end{align*}
Furthermore, for any compact set $K \subset \bbR \setminus D$ and any $\varepsilon > 0$, there exist $C = C(\varepsilon,K) > 0$ and $\eta = \eta(\varepsilon,K) > 0$ such that
\begin{align*}
\mu\left\{ \omega : \left|\frac{1}{n}\log\|A_n^E(\omega)\| - L(E) \right| \geq \varepsilon \right\} \leq Ce^{-\eta n}
\end{align*}
for all $n \in \bbZ_+$ and all $E \in K$.
\end{theorem}

In particular, the nontriviality condition \eqref{eq:NCCAMdef} implies the critical assumptions of F\"urstenberg's theorem and a contractivity property crucial to proving the Large Deviation result.  The boundedness assumption \eqref{cdn:mutildeUnifL2bd} then provides a sufficient regularity property of the cocycle to complete the proof.

\section{Proof of Theorem}

\subsection{A One-Parameter Reformulation of F\"urstenberg's Theorem}




Let $G \subseteq \SL(2,\bbR)$ be a subgroup. We say that $G$ is \emph{contracting} if there exist $g_1,g_2, \ldots \in G$ such that $\|g_n\|^{-1} g_n$ converges to a rank-one operator as $n \to \infty$. Let us say that $G$ is a \emph{type-F subgroup} if $G$ satisfies the following conditions:
\begin{enumerate}
\item $G$ is a contracting subgroup\footnote{Note that if $G$ is contracting, then trivially, $G$ is not compact.} of $\SL(2,\bbR)$.
\item There does not exist $\Lambda \subseteq \bbR\bbP^1$ of cardinality one or two such that $g \Lambda = \Lambda$ for all $g \in G$.
\end{enumerate}

Note that being type-F is monotone in the sense that if $G_1 \subseteq G_2$ are subgroups of $\SL(2,\bbR)$ and $G_1$ is type-F, then $G_2$ is also type-F.

\begin{theorem} \label{t:paramFurst}
Suppose $A,B: \bbC \to \SL(2,\bbC)$ satisfy the following properties:
\begin{align}
& A \text{ and } B \mbox{ are real-analytic functions}\footnotemark  \label{eq:ABanalytic}  \\
& \tr\, A \text{ and } \tr\, B \text{ are nonconstant functions} \label{eq:traceNonConst} \\
& \tr\, A(z) \in [-2,2] \implies z \in \bbR \label{eq:ellipticRealEnergy} \\
& [A(z_0), B(z_0)] \neq 0 \text{ for at least one } z_0 \in \bbC \label{eq:ABdontCommute}.
\end{align}
Then, there is a discrete set $D \subseteq \bbR$ such that the subgroup generated by $A(x)$ and $B(x)$ is a type-F subgroup of $\SL(2,\bbR)$ for any $x \in \bbR \setminus D$.
\footnotetext{That is to say, $A(z)$ and $B(z)$ are analytic functions of $z$ whose entries are real when $\Im \, z = 0$.}
\end{theorem}

\begin{proof}
Using \eqref{eq:ABanalytic} and \eqref{eq:ABdontCommute}, we deduce that there is a discrete set $D_0 \subseteq \bbC$ such that $[A(z),B(z)] \neq 0$ for all $z \in \bbC \setminus D_0$. Combining this with \eqref{eq:traceNonConst} and applying Picard's Theorem to $\tr \, A(\cdot)$, we see that there exists $w \in \bbC \setminus D_0$ such that $\tr \, A(w)\in (-2,2)$. By \eqref{eq:ellipticRealEnergy}, one has $w \in \bbR$, hence $A(w), B(w) \in\SL(2,\bbR)$ by \eqref{eq:ABanalytic}.

Notice $\det \, [A(w),B(w)] \neq 0$. To see this, suppose on the contrary that $\det \, [A(w), B(w)] = 0$. Then, $A(w)$ and $B(w)$ have a common eigenvector; since $A(w)$ is elliptic and $B(w)$ has real entries, this would imply that $A(w)$ and $B(w)$ commute, contradicting $[A(w),B(w)] \neq 0$.

Since $\det [A(z), B(z)]$ is an analytic function of $z$ which does not vanish identically, there is a discrete set $D_1$ such that $\det \, [A(z), B(z)] \neq 0$ for $z \in \bbC \setminus D_1$. Combining this with \eqref{eq:ABanalytic} and \eqref{eq:traceNonConst}, we obtain $D$, a discrete set such that
\[
\det \, [A(z), B(z)] \neq 0,
\quad
\tr\, A(z) \neq 0,
\quad
\tr\,B(z) \neq 0,
\quad z \in \bbC \setminus D.
\]

Let us show that the subgroup generated by $A(x)$ and $B(x)$ is of type F for any $x \in \bbR \setminus D$; to that end, fix $x \in \bbR \setminus D$, write $A = A(x)$, $B = B(x)$, and let $G$ denote the group generated by $A$ and $B$. Since $[A, B] \neq 0$, this implies that $G$ contains a non-elliptic element, $h$ (cf.\ \cite[Theorem 10.4.14]{SimOPUC2}) and hence is contracting (use $g_n = h^n$ to see this). Since $\det \, [A, B] \neq 0$, it follows that $A$ and $B$ have no common eigenvectors. In particular, there cannot be a set $\Lambda \subseteq \bbR\bbP^1$ of cardinality one with $A\Lambda = B \Lambda = \Lambda$.  Suppose instead there exists $\Lambda  \subseteq \bbR\bbP^1$ of cardinality two such that $A\Lambda = B\Lambda = \Lambda$, and denote $\Lambda  = \{\bar{u}_1,\bar{u}_2\}$. Since $\tr \, A \neq 0$,  one cannot have $A\bar{u}_1 = \bar{u}_2$ and $A\bar{u}_2 = \bar{u}_1$, which forces $A \bar{u}_j = \bar{u}_j$ for $j=1,2$. Similarly, $\tr\, B \neq 0$ forces $B \bar{u}_j = \bar{u}_j$. However, this again contradicts the lack of shared eigenspaces between $A$ and $B$. \qedhere

\end{proof}

\subsection{Proof of Theorem}

In our setting, a special case of a classical theorem of F\"urstenberg will yield positive Lyapunov exponent:

\begin{theorem}\label{t:fberg}
For $E \in \bbR$, define  $A^E : \sW \to \SL(2,\bbR)$ as above, let $\nu_E := A^E_\ast\widetilde{\mu}$ be the pushforward of $\widetilde\mu$ under $A^E$, denote by $G_{\nu_E}$ the smallest closed subgroup of $\SL(2,\bbR)$ containing $\supp \nu_E$, and suppose that $\int\log\|M\| \, \dd \nu_E(M) < \infty$. If
$G_{\nu_E}$ is a type-F subgroup of $\SL(2,\bbR)$, then the Lyapunov exponent $L(E) > 0$ is positive.
\end{theorem}
Theorem~\ref{t:fberg} was originally proved by F\"urstenberg under the assumption that $G_{\nu_E}$ is noncompact and strongly irreducible \cite{f}. The sufficient criterion stated here implies strong irreducibility and noncompactness; see, e.g.\ \cite{bougerollacroix}.

Under regularity assumptions on the cocycle one can conclude a uniform Large Deviation Theorem:
\begin{theorem}[Theorem 3.1, \cite{BDFGVWZ}]\label{t:LDT}
Let $K \subset \bbR$ be a compact set, and consider a map $A : K \times \sW \to \SL(2,\bbR)$ such that, for every $E \in K$, $A^E := A(E,\cdot)$ satisfies the assumptions of Theorem \ref{t:fberg}.   Suppose that $A$ also satisfies the following properties:
\begin{equation} \label{cdn:UE} \tag{UnifEq}
\{ E \mapsto A^E(w) : w \in \sW\} \text{ is uniformly equicontinuous;}
\end{equation}
\begin{equation} \label{cdn:UB} \tag{UnifBd}
\exists C>0 \text{ such that } \sup_{E \in K, \ w \in \sW} \|A^E(w)\| \leq C\text{ for $\widetilde\mu$-a.e.\ $w$;}
\end{equation} and, for all $E \in K$,
\begin{equation}\label{cdn:ctrct}\tag{Ctrct}
G_{\nu_E} \text{ is a contracting subgroup of  } \SL(2,\bbR).
\end{equation}
Then, for any $\varepsilon > 0$, there exists $C = C(\varepsilon,K) > 0$, $\eta = \eta(\varepsilon,K) > 0$ such that
\begin{align*}
\mu\left\{ \omega : \left|\frac{1}{n}\log\|A_n^E(\omega)\| - L(E) \right| \geq \varepsilon \right\} \leq Ce^{-\eta n}
\end{align*}
for all $n \in \bbZ_+$ and all $E \in K$.
\end{theorem}

For the 1D continuum Anderson model described above, \eqref{eq:NCCAMdef} implies that $G_{\nu_E}$ is a type-F subgroup of $\SL(2,\bbR)$ away from a discrete set of energies:

\begin{theorem}\label{t:main2}
With notation as above, if $\widetilde{\mu}$ satisfies \eqref{eq:NCCAMdef}, then there is a discrete set $D \subset \bbR$ such that, for any $E \in \bbR\setminus D$, 
$G_{\nu_E}$ is a type-F subgroup of $\SL(2,\bbR)$.

\end{theorem}
Theorem \ref{t:main2} resolves the foremost obstructions to proving Theorem \ref{t:main}.  Indeed,  the boundedness condition \eqref{cdn:mutildeUnifL2bd} and the cocycle structure \eqref{eq:cocycle} of $A^E$ imply \eqref{cdn:UE} and \eqref{cdn:UB} (cf. \cite[Lemma 3.3]{BDFGVWZ}). Thus, Theorem~\ref{t:main} follows from Theorems~\ref{t:fberg}, \ref{t:LDT}, and \ref{t:main2}.

To prove Theorem \ref{t:main2} we will use the following Lemma, which essentially follows from classical inverse spectral theory -- namely the Borg--Marchenko theorem; compare \cite{Bennewitz, Borg, Marchenko}.

\begin{lemma} \label{l:BorgMarchenko}
If $V_1,V_2 \in L^2[0,T)$ and
\begin{equation} \label{eq:allTMsEqual}
A^E(V_1) = A^E(V_2)
\text{ for every }
E \in \bbC,
\end{equation}
then $V_1 = V_2$ Lebesgue almost everywhere on $[0,T)$.
\end{lemma}

\begin{proof}
Denote the $m$-function associated with $V_j$ by $m_j$. That is, taking $\beta$ large enough, then, for every $E \in \bbC \setminus [-\beta,\infty)$, there is a unique (modulo an overall multiplicative constant) solution $u_j = u_j(\cdot,E)$ of
$-u_j'' + V_j u_j = E u_j$
that satisfies a Dirichlet boundary condition at $T$. One then defines the $m$-functions by
\[
m_j(E)
=
\frac{u_j'(0,E)}{u_j(0,E)}.
\]
However, by equality of the cocycles \eqref{eq:allTMsEqual}, it is easy to see $m_1 \equiv m_2$, whence $V_1 \equiv V_2$ (a.e.) by \cite[Theorem~1.1]{Simon1999Annals}.
\end{proof}

We are now ready to prove Theorem \ref{t:main2}, and consequently Theorem \ref{t:main}.

\begin{proof}[Proof of Theorem~\ref{t:main2}]
Let $f_j$ be as in assumption \eqref{eq:NCCAMdef}, denote by $M_j(E)= A^E(f_j)$ the associated monodromies. Notice that Assumptions~\eqref{eq:ABanalytic}, \eqref{eq:traceNonConst}, and \eqref{eq:ellipticRealEnergy} are  satisfied by $M_1$ and $M_2$.  Lemma~\ref{l:BorgMarchenko} and \eqref{eq:NCCAMdef} imply that $[M_1(E), M_2(E)]$ does not vanish identically, i.e. \eqref{eq:ABdontCommute} holds. Thus, the theorem follows from Theorem~\ref{t:paramFurst}. \qedhere

\end{proof}

%
%

\end{document}